\documentclass[a4paper,11pt]{amsart}

\usepackage{amsthm,amsmath,amsfonts,amssymb}

\usepackage{ae}
\usepackage[utf8]{inputenc}
\numberwithin{figure}{section}

\usepackage[colorlinks=true,urlcolor=blue,
citecolor=red,linkcolor=blue,linktocpage,pdfpagelabels,
bookmarksnumbered,bookmarksopen]{hyperref}

\usepackage{mathrsfs,a4wide}
\usepackage{esint}
\usepackage{tikz}
\usetikzlibrary{decorations.pathreplacing}
\usepackage{color}
\usepackage{graphicx}

\usepackage{import}

\usepackage[leqno]{amsmath}

\numberwithin{figure}{section}

\newtheorem{theorem}{Theorem}[section]
\newtheorem{lemma}[theorem]{Lemma}
\newtheorem{proposition}[theorem]{Proposition}

\theoremstyle{definition}
\newtheorem{definition}[theorem]{Definition}

\newtheorem{remark}[theorem]{Remark}
\numberwithin{equation}{section}

\newcommand{\R}{\mathbb{R}}

\newcommand{\Ha}{\mathcal{H}}

\newcommand{\beq}{\begin{equation}}
\newcommand{\eeq}{\end{equation}}

\newcommand{\eps}{\varepsilon}

\newcommand{\pa}{\partial}

\newcommand{\medint}{-\kern -,375cm\int}
\newcommand{\medintinrigo}{-\kern -,335cm\int}

\begin{document}

\title[Stationary sets]{Stationary sets and asymptotic behavior  of the  mean curvature flow with forcing in the plane}

\author{Nicola Fusco, Vesa Julin, Massimiliano Morini}
\address{Dipartimento di Matematica e Applicazioni,
Universit\`{a}  di Napoli "Federico II", Napoli, Italy}
\email{n.fusco@unina.it}
\address{Jyv\"askyl\"an Yliopisto, Matematiikan ja Tilastotieteen Laitos, Jyv\"askyl\"a, Finland.}
\email{vesa.julin@jyu.fi}
\address{Universit\`a di Parma, Dipartimento di Scienze Matematiche, Fisiche e Informatiche, Parma, Italy}
\email{massimiliano.morini@unipr.it}


\begin{abstract} 
We consider the flat flow solutions of the mean curvature equation with a forcing term in the plane. We prove that  for every constant forcing term the stationary sets are given by a finite union of disks with equal radii and disjoint closures. On the other hand for every bounded forcing term tangent disks are never stationary. Finally in the case of an asymptotically constant forcing term we show that the only possible long time limit sets are given by disjoint unions of disks with equal radii and possibly tangent.
\end{abstract}

\maketitle

\section{Introduction}
Mean curvature flow is one of the simplest and yet most interesting geometric evolution equation. In order to deal with formation of singularities or rough initial data several notions of generalized solutions have been proposed. Among them we mention Brakke's solutions in the varifold sense \cite{Bra}, level-set solutions in the viscosity sense \cite{CGG}, \cite{ES}, De Giorgi's minimal barriers \cite{DG} and the flat flows solutions constructed by the minimizing movements method \cite{ATW}, \cite{LS}. Each method has its own advantages and drawbacks. For instance Brakke's theory fails to provide unique solutions, but yields a satisfactory partial regularity theory, see also \cite{KT}. On the contrary, the viscosity level-set method provides uniqueness and   global existence, but it is not so convenient as far as regularity is concerned. Indeed in this framework one may construct singular solutions where the evolving hypersurfaces become sets with nonempty interior, the so called fattening phenomenon. This phenomenon can occur even if the initial set is regular after a positive time, see \cite{BP}. De Giorgi's minimal barriers provide essentially the same solutions as the level-set method, see \cite{BN}; within this approach the fattening phenomenon is related to the fact that  minimal and  maximal solutions may be different, see \cite{BP}. Flat flow solutions are also defined globally in time. They are always given by evolving boundaries of sets and may not be unique whenever the level-set solution experiences the fattening phenomenon. However, level-set solutions, De Giorgi's minimal barriers and  flat flows all coincide with the classical solutions as long as the latter exist.

In this paper we focus on the flat flow approach for the mean curvature equation with a time dependent forcing term in the plane, i.e.,
 \beq
\label{flow}
V_t = -k_{E_t} + f(t) \qquad \text{on } \, \pa E_t
\eeq
with an arbitrary initial datum under the assumption that the forcing term $f$ is uniformly bounded, i.e., 
\beq 
\label{f1}
\sup_{t \geq 0} |f(t)| \leq C_0.
\eeq
Here $k_{E_t}$ stands for the curvature of the boundary of $E_t$ with respect to the orientation given by the outward normal.
For the precise definition of flat flow see the beginning of Section~\ref{due}.

 The existence of flat flow solutions for the equation \eqref{flow} in any dimension and their relations with the De Giorgi' barriers and the level-set solutions has been investigated in \cite{CN}. In this paper we further elaborate on the properties of flat flows solutions in two dimensions focusing on the following issues: how the flat flow selects a solution when the fattening phenomenon occurs, the characterization of sets that are stationary  when $f$ is constant and the long time behavior of solutions.

\subsection{Flat flow as a selection principle} Here we consider a particular situation where the initial set is given by two tangent disks of equal radii $D_r(x_1)$ and $D_r(x_2)$. It is well known that in this example the level-set solution develops instantaneously a nonempty interior. When $f(t)\equiv1/r$ the minimal barrier solution of \eqref{flow} is stationary, while the maximal barrier solution becomes a connected set containing a ball centered at the origin with a time dependent radius, see \cite{BP}. It is an interesting problem to look for a selection principle among the possible admissible behaviors. One such principle can be obtained by adding to the forcing term a small stochastic perturbation. This has been investigated in \cite{DLN} where the perturbation considered is of the form $\eps\,dW$, with $W$  a standard Brownian motion. The authors show that when $\eps$ goes to zero the corresponding motion converges with probability $1/2$ to the maximal barrier solution and with probability $1/2$ to the minimal one. In this paper we prove that any flat flow instantaneously connects the two tangent disks with a thin neck and keeps enlarging the neck at least for a short time interval, thus showing that the flat flow somehow picks the behavior of the maximal barrier solution. The precise statement is as follows.

\begin{theorem}
\label{thm1}
Let $E_0 \subset \R^2$ be a union of two tangent disks $E_0 = D_r(x_1) \cup D_r(x_2)$  and let $(E_t)_t$ be  a  flat flow of \eqref{flow} starting from $E_0$ and assume that \eqref{f1} holds. There exist $\delta>0$, $\eta>0$ and $c>0$ such that for every $t \in (0,\delta)$ the set   $E_t$ contains a  dumbbell shaped 
simply connected set which in turn contains the disks $D_{\eta r}(x_1)$ and $D_{\eta r}(x_2)$ and a ball centered at the origin of radius $t$.  In particular for every $t \in (0,\delta)$ 
\[
 |E_t \setminus E_0| \geq c\,t^3. 
\]
\end{theorem} 
This theorem is also relevant for the second issue we want to deal with, i.e., the characterization of stationary sets, as it shows that the union of two equal tangent disks is not stationary for the flat flow.

\subsection{Characterization of stationary sets}   When the forcing term $f\equiv c_0$ equation \eqref{flow} can be regarded as the gradient flow of the following energy
\beq\label{energia}
\mathcal E(E)=P(E)-c_0|E|,
\eeq
where $P(E)$ stands for the perimeter of $E$ and $|E|$ for its Lebesgue measure. Therefore one might think that $E_0$ is stationary for the flow if and only if it is critical for the energy \eqref{energia}, i.e., it satisfies $k_{E_0}=c_0$ on $\pa E_0$ in a weak sense. Indeed if $E_0$ is stationary then it also critical, while the converse  is certainly true when $E_0$ is smooth, i.e., is given by a union of finitely many disks with equal radii and mutually disjoint closures (see \cite{DM} for a characterization of critical sets in any dimension, even in the nonsmooth case).  However, Theorem~\ref{thm1} shows that the two notions  do not coincide since the union of two tangent disks of equal radii is critical as it has constant mean curvature in the weak sense, but not stationary.  Here we show that a set $E$ is \emph{stationary} for the  flow \eqref{flow} when  $f \equiv c_0$ if and only if  it is a union of disks with radius $r = 1/c_0$  with positive distance to each other. More precisely we have the following.

\begin{theorem}
\label{corollary}
Assume  $E_0 \subset \R^2$ is a bounded  set of finite perimeter.   Then $E_0$ is stationary (see Definition \ref{stationary}) for the  flow \eqref{flow} with  $f \equiv c_0$ if and only if there are points  $x_1, \dots, x_N$  such that  $|x_i -x_j|  > 2r$ for $i \neq j$,  with $r = 1/c_0$, and  
\[
E_0 = \bigcup_{i =1}^N D_r(x_i).
\]
\end{theorem}
The fact that any stationary set is a union of disjoint disks follows from a sharp quantitative version of the Alexandrov theorem in the plane, see Lemma~\ref{alex}, while the fact that the disks must be at positive distance apart is a consequence of Theorem~\ref{thm1}.

We remark that the same type of classification holds true in the framework of level-set solutions, as recently shown in \cite[Theorem 4.7]{GMT}. The general $n$-dimensional case remains open also for the viscosity solutions, see \cite{GTZ}.

\subsection{Long time behavior}   We now address the long time behavior of the flat flow under the assumption
that the forcing term is asymptotically  constant, namely that it satisfies 
 \beq 
\label{f2}
\int_0^\infty |f(s) -c_0|^2\, ds <\infty.
\eeq

 In the next theorem our goal is to   characterize the possible limit sets and we show in particular that the asymptotically stationary sets are given once again by a union of disjoint disks, which however can be tangent. Precisely we  show that either, up to a diverging sequence $t_j$ of times, the area of $E_{t_j}$ blows up or  the sets $(E_t)_t$ converge up to a translation  in the Hausdorff sense to a disjoint union of  disks with equal radii. 

\begin{theorem}
\label{thm2}
Assume  $E_0 \subset \R^2$ is a bounded  set of finite perimeter. Let $(E_t)_t$ be a flat flow of \eqref{flow} starting from $E_0$ and  assume \eqref{f1} and  \eqref{f2} with $c_0 >0$, and  
\[
\sup_{t >0} |E_t| < \infty.
\]
Then there exist $N\in\mathbb N$ and $x_i(t):(0,+\infty)\to\R^2$, with $i=1,\dots,N$ and $|x_i(t) -x_j(t)| \geq 2/c_0$ for $i \neq j$, such that, setting $F_t = \cup_{i=i}^N D_{1/c_0}(x_i(t))$ 
\[
\lim_{t\to\infty}\,\sup_{x\in E_t\Delta F_t}d_{\pa F_t}(x)=0.
\]  
\end{theorem}
We stress here the fact that the initial set $E_0$ in the above theorem is an arbitrary bounded set of finite perimeter without further regularity assumption.
 It is plausible that in Theorem~\ref{thm2} the convergence holds not just up to translation.
 
 Previous results dealt with special classes of sets in any dimension such as convex or star-shaped initial sets, see for instance \cite{BCCN} and \cite{KK}. We also mention \cite{MoPoSpa} where the long-time behavior of the discrete  Euler implicit scheme for the volume preserving mean curvature flow is addressed for any arbitrary bounded initial set with finite perimeter. 
 The long time behavior of the forced mean curvature flow in the context of viscosity level-set solutions was also investigated in \cite{GTZ} and \cite{GMT} where it is shown that under certain assumptions the solutions converge to a stationary solution of the level-set equation. The problem of classifying the latter is open in general.

We now show that it is indeed possible to obtain as a limit of the flow \eqref{flow} a union of essentially disjoint disks such that at least two of them are tangent.  To this end we take  $G$ to be the ellipse
\[
G = \{ (x_1,x_2) \in \R^2 : \,  a^2 x_1^2 + x_2^2 < 1 \} \qquad \text{with } \, a>1
\]
and we show the following theorem.
\begin{theorem}
\label{thm3}
Let $e_1 = (1,0)$ and $G$ as above. Denote by $\rho= \frac{1}{\sqrt{a}}$ the radius such that $|D_\rho| = |G|$. The volume preserving mean curvature flow $(E_t)_t$, starting from 
\[
E_0 = (G - \rho e_1 ) \cup  (G + \rho e_1 ),
\] 
is well defined in the classical sense for all $t >0$ and converges exponentially fast to the union of two tangent disks
\[
E_t \to  (D_\rho - \rho e_1 ) \cup  (D_\rho + \rho e_1 ). 
\]
\end{theorem} 
Note that Theorem \ref{thm3} shows that a flat flow of \eqref{flow} may converge to tangent disks.  Indeed the classical solution of the flow in Theorem \ref{thm3} is well defined and  smooth for all times and we may write it in the form \eqref{flow} with $f(t) =\medintinrigo_{\pa E_t}k_{E_t}$ and the flat flow agrees with it.  Moreover, by the exponential convergence we have that $f(t)$ satisfies \eqref{f2}.  

We note that in  Theorem \ref{thm3} the flow $(E_t)_t$ remains smooth and diffeomorphic to a union of two disks. Only the limit set is non-smooth.

\begin{figure}
\begin{tikzpicture}
\begin{scope}[yscale=.6,xscale=.6]
\clip (-7, -5) rectangle (7, 5); 

\filldraw[fill=white!80!black] (3,0) ellipse (2 and 4.5); 
\draw (3,0) circle (3); 
\filldraw[fill=black] (3,0) circle (.05);
\draw (3,0.4) node {$(\rho,0)$}; 

\begin{scope}[yscale=1,xscale=-1]  
\filldraw[fill=white!80!black] (3,0) ellipse (2 and 4.5); 
\draw (3,0) circle (3); 
\filldraw[fill=black] (3,0) circle (.05);
\draw (3,0.4) node {$(-\rho,0)$};
\end{scope}

\draw[->] (-7,0) -- (7,0); 
\draw (6.8,0.5) node 
                       {$x_1$};
\draw[->] (0,-5) -- (0,5); 
\draw (0.5,4.8) node 
                       {$x_2$};
             
\end{scope}
\end{tikzpicture}
\caption{The union of two ellipses converges to the union of two tangent disks.}
\end{figure}
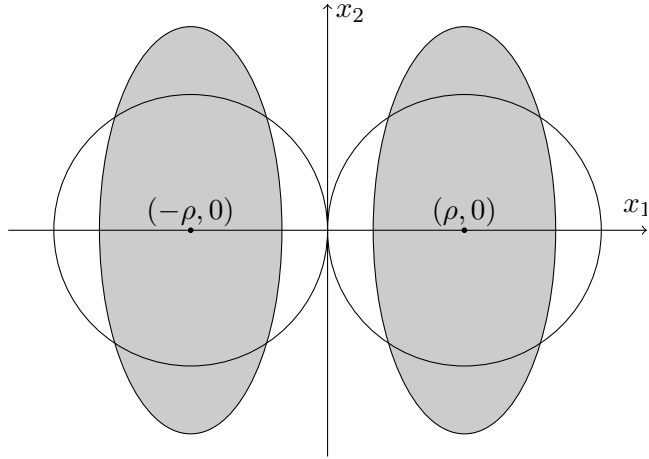

\section{Notation and Preliminary results}\label{due}
Since the results of this section hold in any dimension we state them in full generality and we will go back to the planar case in the next sections.

Given a set $A \subset \R^n$   the distance function $d_A : \R^n \to [0,\infty)$ is defined as usual
\[
d_A(x) := \inf_{y \in A} |x - y|
\] 
and we denote the signed distance function by $\bar{d}_A : \R^n \to \R$, 
\[
\bar{d}_A(x) := \begin{cases} - d_{\R^n \setminus A}(x) , \,\,\text{for }\, x \in A \\
d_A(x) , \,\, \text{for }\, x \in \R^n \setminus A.  \end{cases} 
\]
Then clearly it holds $d_{\pa A} = |\bar d_A|$.

For a set of finite perimeter  $E \subset \R^n$  we denote its perimeter by $P(E)$ and recall that for regular enough set it holds $P(E) = \Ha^{n-1}(\pa E)$ \cite{AFP, Ma}.  For a measurable set $|E|$ denotes its Lebesgue measure. We denote by $H_E$ the sum of the principal  curvatures of $E$, while in the planar case we write $k_E$. We denote the disk with radius $r$ centered at $x$ by $D_r(x)$ and in the higher dimensional case we write $B_r(x)$ instead.

We consider  solutions of \eqref{flow} constructed via the minimizing movement scheme.  We fix a small time step $h>0$ and a bounded  set of finite perimeter $E_0 \subset \R^n$,  which is our initial set $E^{h,0} = E_0$. We obtain a sequence of set 
$(E^{h,k})_{k=1}^\infty$ by iterative minimizing procedure, where $E^{h,k+1}$ is a minimizer of the functional $\mathcal{F}_k(E; E^{h,k})$ defined as 
\beq
\label{min mov}
\mathcal{F}_k(E; E^{h,k}) =   P(E) + \frac{1}{h} \int_E \bar{d}_{E^{h,k}} \, dx  - \bar f(kh) |E|,
\eeq
where $\bar{d}_{E^{h,k}}$ is the signed distance defined above and $\bar f(kh) = \fint_{kh}^{(k+1)h} f(s)\, ds$.  We define the approximate flat  flow $(E_t^h)_{t > 0}$ by
\beq
\label{disc flow}
E_t^h = E^{h,k},  \qquad \text{for } \,  (k-1)h < t \leq k h 
\eeq
and we set $\bar f(t)=\bar f(kh)$ for $(k-1)h < t \leq k h$.  Any  cluster point of $E_t^h$ as $h$ goes to zero is called a flat flow for the equation \eqref{flow}. 

We warn the reader that in the above definition it is understood that we identify $E^{h,k}$ with its set of its points of density $1$ so that there is no ambiguity in the definition of $\bar{d}_{E^{h,k}}$.

Recall that if $E_0$ and $f$ are smooth then any flat flow coincide with the classical solution of \eqref{flow} as long as the latter remains smooth, see \cite{CN}.

In general, the problem \eqref{min mov} does not admit a unique minimizer and thus there is no unique way to define the approximate flat  flow $(E_t^h)_{t > 0}$. Also  the flat flow may not be unique when fattening occurs. However, as we mentioned in the introduction,  in the case when the initial set  and the forcing term are  smooth,  the flat flow is unique   for a short time interval and agrees with the classical solution. 

Even if there is no uniqueness, the approximate flat flow  satisfies the following weak comparison principle, see for instance the proof of Lemma~6.2 in \cite{CMP}.
\begin{proposition}
\label{compa prin}
Assume  $f_1, f_2:[0,\infty) \to \R$ satisfy \eqref{f1}. Let   $E_0, F_0$ be two bounded sets of finite perimeter and let  $(E_t^h)_t$ be an approximate flat flow  with forcing term $f_1$ starting from $E_0$ and $(F_t^h)_t$ an  approximate flat flow with  forcing term $f_2$ starting from $F_0$. 
\begin{itemize}
\item[(i)] If  $F_0 \subset E_0$ and $f_1 > f_2$, then for every $t>0$ it holds $F_t^h \subset E_t^h$. 
\item[(ii)]  If  $E_0 \subset \R^n \setminus  F_0$ and $-f_2 > f_1$, then for every $t>0$ it holds $E_t^h \subset \R^n \setminus  F_t^h$.
\end{itemize}
\end{proposition}

We need preliminary results on  the structure of the approximate flat flow constructed via \eqref{min mov}. We  note that if $E^{h,k+1}$ is a minimizer of $\mathcal{F}_{k}(\cdot, E^{h,k})$ then it is a  $\Lambda$-minimizer of the perimeter, see for instance \cite{MSS}, with $\Lambda \leq C/h$, see \cite{Ma} for the definition of $\Lambda$-minimizer. Then it follows that  $\partial E^{h,k+1}$ is $C^{1,\alpha}$-regular for all $\alpha\in(0,1)$ up to a singular set $\Sigma$ with Hausdorff dimension  at most $n-8$, see \cite{Ma}. Then  the Euler-Lagrange equation
\beq
\label{euler}
\frac{\bar{d}_{E^{h,k}}}{h} = -H_{E^{h,k+1}} + \bar  f(kh) \qquad \text{on } \, \pa E^{h,k+1}\setminus\Sigma,
\eeq
which holds in the weak sense, implies that  $\partial E^{h,k+1}\setminus\Sigma$ is $C^{2,\alpha}$-regular and satisfies \eqref{euler} in the classical sense.

\begin{lemma}
\label{distance}
Assume that $(E^{h,k})_k$ is a sequence obtained via minimizing movements \eqref{min mov} starting from a bounded set of  finite perimeter $E_0$  and assume that the forcing term satisfies \eqref{f1}. Then there is a constant $C_1$ such that for every $k = 0,1,2, \dots $  
\[
\sup_{x \in E^{h,k+1} \Delta E^{h,k}} d_{\pa E^{h,k}}(x) \leq C_1 \sqrt{h}.
\]
Moreover, there are constants $C_2>1$ and $c_1>0$ such that  for every $k =1,2,3, \dots$ it holds 
\[
|E^{h,k+1} \Delta E^{h,k}| \leq C_2 \left( l P( E^{h,k}) + \frac{1}{l} \int_{E^{h,k+1} \Delta E^{h,k}} |\bar d_{E^{h,k}}(x)| \, dx \right)
\]
for any $0 < l < c_1 \sqrt{h} $.
\end{lemma}

\begin{proof}
The first claim follows from the argument of the  proof  of  \cite[Proposition 3.2]{MSS} and thus we omit it. The second claim follows from an argument similar to \cite[Proposition 3.4]{MSS} and we only sketch it. 
We write
\[
|E^{h,k+1} \Delta E^{h,k} |  =  | \{x \in E^{h,k+1} \Delta E^{h,k}  : |\bar d_{E^{h,k}}(x)|\geq l\}|  + |\{x \in E^{h,k+1} \Delta E^{h,k}  : |\bar d_{E^{h,k}}(x)| < l\}|.
\]
We estimate the first term as
\[
| \{x \in E^{h,k+1} \Delta E^{h,k}  : |\bar d_{E^{h,k}}(x)|\geq l\} | \leq \frac{1}{l} \int_{E^{h,k+1} \Delta E^{h,k}} |\bar d_{E^{h,k}}(x)| \, dx.
\]

For the second term we use  Vitali covering theorem to choose a finite family of disjoint balls $(B_l(x_i))_{i=1}^N$, with $x_i \in \partial E^{h,k}$, such that 
\[
\{x \in  \R^n : |\bar d_{E^{h,k}}(x)| < l\} \subset \cup_{i=1}^N B_{5l}(x_i).
\]
Since $E^{h,k}$ is a minimizer of $\mathcal{F}_k(E; E^{h,k-1})$, we have  the density estimates   \cite[Corollary 3.3]{MSS}. Thus by the relative isoperimetric inequality we have for every $i = 1, \dots, N$
\[
|B_l(x_i)| \leq C (\Ha^{n-1}(\pa E^{h,k} \cap B_l(x_i))^{\frac{n}{n-1}}\leq C l \, \Ha^{n-1}(\pa E^{h,k} \cap B_l(x_i)).
\]
Therefore
\[
\begin{split}
|\{x \in E^{h,k+1} \Delta E^{h,k}  : |\bar d_{E^{h,k}}(x)| < l\}| &\leq  \sum_{i=1}^N |B_{5l}(x_i)| \leq 5^n \sum_{i=1}^N |B_{l}(x_i)| \\
&\leq Cl \sum_{i=1}^N \Ha^{n-1}(\pa E^{h,k} \cap B_l(x_i)) \leq C l \, P(E^{h,k} ).
\end{split}
\]
\end{proof}

In the next proposition we list useful properties of the flow in the case when the forcing term satisfies only \eqref{f1}.
\begin{proposition}
\label{MSS}
Let  $(E_{t}^h)_t$ be an approximate flat flow starting from a bounded set of  finite perimeter $E_0$  and assume  that the forcing term satisfies \eqref{f1}.  Then the following hold:
\begin{itemize}
\item[(i)] For every $T>0$ there is $R_T>0$ such that $E_t^h \subset B_{R_T}$ for every $t \leq T$. 
\item[(ii)] There is $C_3$, depending only on $E_0$ and $f$,  such that for every $T >0$ it holds  
\[
P(E_T^h) \leq C_3^{1+T}
\]
for $h$ sufficiently small.
\item[(iii)]  For every $h < s <t <T$ with $t-s >h$ and $h$ sufficiently small, it holds $|E_t^h \Delta E_s^h| \leq C_T \sqrt{t-s}$, where the constant $C_T$ depends on $T$.
\item[(iv)]  There exists a subsequence $(h_l)_l$ converging to zero such that $(E_t^{h_l})_t$ converges to a flat flow $(E_t)_t$ in $L^1$ in space and locally uniformly in time, i.e., for every $T$ 
\[
\sup_{h_l< t \leq T }|E_t^{h_l} \Delta E_t| \to 0 \qquad \text{as } \, h_l \to 0.
\]
\end{itemize}
\end{proposition}

\begin{proof}
The claim (i)  follows by applying   Proposition \ref{compa prin} to $E_t^h$ and $F_t^h$, where the latter is approximate flat flow starting from $B_R$, such that $E_0 \subset B_R$, and with constant forcing term $f_2 \equiv \sup_t f(t) +1$. Then $E_t^h \subset F_t^h$.  It is easy to check that the sets $(F_t^h)_{t \leq T}$ are  balls whose radii satisfy  $r(t) \leq C(1+T)$ for $t \leq T$.  

Let us prove  (ii). By the minimality of $E^{h,k+1}$ we have $\mathcal{F}_{k}( E^{h,k+1}; E^{h,k}) \leq \mathcal{F}_{k}( E^{h,k}; E^{h,k})$ which implies
\[
P(E^{h,k+1}) + \frac{1}{h} \int_{E^{h,k+1}} \bar{d}_{E^{h,k}} \, dx  - \bar f(kh) |E^{h,k+1}| \leq P(E^{h,k}) + \frac{1}{h} \int_{E^{h,k}} \bar{d}_{E^{h,k}} \, dx  - \bar f(kh) |E^{h,k}|.
\]
We write this as
\beq \label{prop MSS 1}
\frac{1}{h} \int_{E^{h,k+1} \Delta E^{h,k}} |\bar{d}_{E^{h,k}}| \, dx  + P(E^{h,k+1})  \leq P(E^{h,k})  + \bar f(kh) (|E^{h,k+1}| - |E^{h,k}|).
\eeq
By \eqref{f1} we simply estimate $\bar f(kh) (|E^{h,k+1}| - |E^{h,k}|) \leq C_0 |E^{h,k+1} \Delta E^{h,k}|$. Then we use the second statement in Lemma \ref{distance} with $l = \hat C h$, where $\hat C$ is a large constant to deduce
\[
|E^{h,k+1} \Delta E^{h,k}| \leq C h \, P(E^{h,k}) + \frac{1}{2C_0 h} \int_{E^{h,k+1} \Delta E^{h,k}} |\bar{d}_{E^{h,k}}| \, dx.
\]
Therefore we deduce from these two inequalities and from \eqref{prop MSS 1} that 
\beq \label{prop MSS 2}
\frac{1}{2h} \int_{E^{h,k+1} \Delta E^{h,k}} |\bar{d}_{E^{h,k}}| \, dx + P(E^{h,k+1})  \leq (1+ C \, h ) P(E^{h,k}).
\eeq
By iterating the inequality $P(E^{h,k+1})  \leq (1+ C \, h ) P(E^{h,k})$ we get
\[
P(E^{h,k}) \leq (1+ C \, h )^{k-1} P(E^{h,1}) =  \left((1+ C \, h )^{1/h}\right)^{(k-1)h} P(E^{h,1}) \leq C^{(k-1)h}\,  P(E^{h,1}).
\]
Finally we  use \eqref{prop MSS 1} for $k=0$ and have
\beq \label{prop MSS 3}
P(E^{h,1})  \leq P(E_{0})  + \bar f(h) (|E^{h,1}| - |E_{0}|).
\eeq
By  (i)  we may estimate $|E^{h,1}| \leq |B_{2R}|$ for $h$ sufficiently small, where we recall that $B_R$ is the ball containing $E_0$. Therefore $ P(E^{h,1}) \leq P(E_0) + C$ and we obtain the claim (ii)

The claim (iii) follows from argument similar to \cite[Proposition 3.5]{MSS} so we only point out the main differences.  Let $k, m$ be such that   $s \in (kh, (k+1)h]$ and $t \in ((k+m)h, (k+m+1)h]$. Note that $mh \leq 2(t-s)$.  We may estimate the quantity $|E_t^h \Delta E_s^h| $ by applying the second statement of Lemma \ref{distance} with $l = c_1\frac{h}{2\sqrt{t-s}}$,  \eqref{prop MSS 2} and the part (ii) to get 
\[
\begin{split}
|E_t^h \Delta E_s^h| &\leq \sum_{i=1}^m |E^{h,k+i+1} \Delta E^{h,k+i}|\\
&\leq\sum_{i=1}^m  C \left( \frac{h}{\sqrt{t-s}} P( E^{h,k+i}) + \frac{\sqrt{t-s}}{h}\int_{E^{h,k+i+1} \Delta E^{h,k+i}} |\bar d_{E^{h,k+i}}(x)| \, dx \right)\\
&\leq \sum_{i=1}^m  C \left( \frac{h}{\sqrt{t-s}}P( E^{h,k+i})  + \sqrt{t-s}\big((1+ C \, h ) P( E^{h,k+i}) - P( E^{h,k+i+1}) \big)  \right)\\
&\leq  C\sqrt{t-s} \, \sup_{t\leq T}P(E_t^h) +  \sqrt{t-s} P(E^{h,k+1}) \leq C_T \sqrt{t-s}.
\end{split}
\]

Similarly (iv) follows from the proof of \cite[Theorem 2.2]{MSS}.
\end{proof}

When in addition we assume that the forcing term satisfies  \eqref{f2} we obtain  estimates which are more uniform with respect to time. To this aim we define the following quantity which plays the role of the energy
\beq
\label{energy}
\mathcal{E}(E)  := P(E) -c_0 |E|,
\eeq
where $c_0$ is the constant appearing in \eqref{f2}.
\begin{proposition}
\label{MSS2}
Let  $(E_{t}^h)_t$ be an approximate flat flow starting from a bounded set of  finite perimeter $E_0$  and assume  that the forcing term satisfies \eqref{f1} and \eqref{f2}.  Then, if $h$ is sufficiently small, the following hold:
\begin{itemize}
\item[(i)] For every $\eps>0$ there is $T_\eps$ such that for every  $T_\eps < T_1 < T_2$, with $T_2 \geq T_1+h$, we have the following dissipation inequality  
\[
c \int_{T_1}^{T_2}\int_{\pa E_t^h} (H_{E_t^h} - \bar f(t-h))^2 \, d\Ha^{n-1}dt  + \mathcal{E}(E_{T_2}^h )\leq  \mathcal{E}(E_{T_1-h}^h) + \eps \sup_{T_1-h \leq t \leq T_2} P(E_{t}^h).
\]
\item[(ii)] If $\sup_{t \geq 0} | E_t^h| < \infty$, then $\sup_{t \geq 0} P(E_t^h) < \infty$.
\item[(iii)]  If $\sup_{t \geq 0} | E_t^h| < \infty$, there exists a constant $C_4$ such that  $|E_t^h \Delta E_s^h| \leq C_4 \sqrt{t-s}$ for every $h < s <t$ with $t-s >h$.

\end{itemize}
\end{proposition}

\begin{proof}
To prove (i) we begin with \eqref{prop MSS 1}. This time we estimate the last term in  \eqref{prop MSS 1} as
\[
\bar f(kh) (|E^{h,k+1}| - |E^{h,k}|) \leq  c_0(|E^{h,k+1}| - |E^{h,k}|) +  |\bar f(kh) -c_0| \, |E^{h,k+1} \Delta E^{h,k}|.
\]
We use the second estimate in Lemma \ref{distance} with $l = \hat C\, |\bar f(kh) -c_0| h$, where $\hat C$  is a large constant and $h$ is sufficiently small, to deduce
\[
 |\bar f(kh) -c_0| \, |E^{h,k+1} \Delta E^{h,k}| \leq C |\bar f(kh) -c_0|^2 h \, P(E^{h,k}) + \frac{1}{2 h} \int_{E^{h,k+1} \Delta E^{h,k}} |\bar{d}_{E^{h,k}}| \, dx.
\]
Therefore we have by \eqref{prop MSS 1} that 
\[
\frac{1}{2 h} \int_{E^{h,k+1} \Delta E^{h,k}} |\bar{d}_{E^{h,k}}| \, dx + \mathcal{E}(E^{h,k+1} )\leq  \mathcal{E}(E^{h,k}) + C |\bar f(kh) -c_0|^2 h \, P(E^{h,k}),
\]
where $\mathcal{E}$ is defined in \eqref{energy}. 

Let us fix $\eps>0$. Since we assume \eqref{f2}, there exists $T_\eps$ such that 
\beq \label{prop MSS2 1}
 \int_{T_\eps}^\infty (f(t) -c_0)^2\, dt \leq \frac{\eps}{C},
\eeq
 where $C$ is a constant to be chosen later. Let $T_2 > T_1 > T_\eps$  and 
let $j, m$ be such that   $T_1 \in (jh, (j+1)h]$ and $T_2 \in ((j+m)h, (j+m+1)h]$. We iterate the previous inequality from $k= j-1$ to $k = j+m-1$ and obtain 
\beq \label{prop MSS2 2}
\begin{split}
\sum_{k=j}^{j+m} \frac{1}{2 h} \int_{E^{h,k+1} \Delta E^{h,k}} &|\bar{d}_{E^{h,k}}| \, dx + \mathcal{E}(E_{T_2}^h )  \\
&\leq \mathcal{E}(E_{T_1-h}^h) +C\left(  \sup_{T_1-h \leq t \leq T_2} P(E_{t}^h) \right)  \left( \int_{T_1-h}^{T_2} |\bar f(t) -c_0|^2 \,dt  \right) \\
&\leq \mathcal{E}(E_{T_1-h}^h) +C\left(  \sup_{T_1-h \leq t \leq T_2} P(E_{t}^h) \right)  \left( \int_{T_\eps}^{\infty} |f(t) -c_0|^2 \,dt  \right) \\
&\leq \mathcal{E}(E_{T_1-h}^h) + \eps \left(  \sup_{T_1-h \leq t \leq T_2} P(E_{t}^h) \right) ,
\end{split}
\eeq
where the last inequality follows from \eqref{prop MSS2 1}.

Arguing as in the proof of \cite[Lemma 3.6]{MSS}, we deduce that there is a constant $c>0$, depending only on the dimension,  such that  
\[
c\, h \int_{\pa E^{h,k+1}}  \left( \frac{\bar{d}_{E^{h,k}}}{h}\right)^2 \, d \Ha^{n-1} \leq  \int_{E^{h,k+1} \Delta E^{h,k}} \frac{|\bar{d}_{E^{h,k}}|}{h} \, dx.
\]
Therefore by the Euler-Lagrange equation \eqref{euler} we have
\[
\begin{split}
\sum_{k=j}^{j+m} \frac{1}{h} \int_{E^{h,k+1} \Delta E^{h,k}} |\bar{d}_{E^{h,k}}| \, dx &\geq c \sum_{k=j}^{j+m} h \int_{\pa E^{h,k+1}}  \left( \frac{\bar{d}_{E^{h,k}}}{h}\right)^2 \, d \Ha^{n-1}\\
&= c \sum_{k=j}^{j+m} h \int_{\pa E^{h,k+1}}  \left( H_{E^{h,k+1}} - \bar f(kh)\right)^2 \, d \Ha^{n-1}\\
&\geq c \int_{T_1}^{T_2}  \int_{\pa E_{t}^h}  \left( H_{E_{t}^h} - \bar f(t-h)\right)^2 \, d \Ha^{n-1} dt . 
\end{split}
\]
Thus we have the claim (i) by \eqref{prop MSS2 2}.

To show (ii) we fix $0 <\eps < 1/2 $,  $T > T_\eps$ and apply  the  part (i) with  $T_1 = T_\eps+h$ and $T_1+h<T_2 = t \leq T$ to deduce 
\[
\mathcal{E}(E_{t}^h )\leq  \mathcal{E}(E_{T_\eps}^h) + \eps \sup_{T_\eps \leq s \leq T} P(E_{s}^h).
\]
We recall that $\mathcal{E}(E) = P(E) - c_0 |E|$ and that we assume $\sup_{t> 0}|E_t^h| <\infty$. Therefore from the above inequality, recalling that $P(E_t^h)\leq C_\eps$ for all $t<T_\eps+1$ by   Proposition~\ref{MSS} (ii), we get 
\[
P(E_{t}^h ) \leq C_\eps + c_0\sup_{t> 0}|E_t^h|+\eps \sup_{T_\eps \leq s \leq T} P(E_{s}^h)
\]
for every $T_\eps< t\leq T$. Thus, since $\eps < 1/2$ we deduce  that
\[
\sup_{T_\eps \leq t \leq T} P(E_{t}^h) \leq 2\big(C_\eps + c_0\sup_{t> 0}|E_t^h|\big). 
\]
The claim (ii) follows from the fact that $T$ was arbitrary. 

Finally the proof of (iii) follows from the proof of Proposition \ref{MSS} (iii), noticing that now the constant $C_T$ is in fact independent on $T$ thanks to the bound on the perimeters provided by (ii).
\end{proof}

\begin{remark}
\label{rem energy}
If $(E_t^h)_t$, $E_0$ and $f$ are as  in  Proposition \ref{MSS2},  and if we assume 
\[
\sup_{t \geq 0} | E_t^h| \leq C,
\] 
then Proposition (i) and (ii) imply that the energy $\mathcal{E}(E_t^h)$ is asymptotically almost decreasing. More precisely, for every $\eps>0$ there is $T_\eps$ such that for $t>s>T_\eps$ it holds
\beq\label{remenergy1}
\mathcal{E}(E_t^h) \leq \mathcal{E}(E_s^h) + C \eps,
\eeq
with $T_\eps$ and $C$ independent of $h$. 
This inequality implies in particular that there exists
$$
\lim_{t\to+\infty}\mathcal{E}(E_t^h) .
$$
Moreover, from the proof of  Proposition~\ref{MSS2} we have also that if  $h$ is sufficiently small and $\sup_{0<t <T} | E_t^h| \leq C$ for some $T>0$, then there exists a constant $\tilde C$, independent of $h$, such that $\sup_{0<t <T} P(E^h_t) \leq\tilde C$.
\end{remark}

\section{Stationary sets and proof of Theorem \ref{thm1}}

In this section we go back to the two dimensional setting. We study critical sets of the isoperimetric problem and  stationary sets for the flow \eqref{flow}. A set of finite perimeter $E$ is \emph{critical} for the isoperimetric problem if its distributional mean curvature is constant. 

We   define \emph{stationary sets} for the equation \eqref{flow} as follows.
\begin{definition}
\label{stationary}
Assume that the forcing term $f$ in  \eqref{flow} is constant, i.e., $f \equiv c_0$. A set of finite perimeter $E_0$ is \emph{stationary}  if for any  flat flow $(E_t)_t$ starting from $E_0$ it holds
\[
\sup_{0\leq t \leq T} |E_t \Delta E_0| = 0
\]  
for every $T>0$.
\end{definition}

We begin by proving the sharp  quantitative version of the Alexandrov's theorem in the plane. 

\begin{lemma}
\label{alex}
Let $M>0$ and let $E  \subset  \R^2$ be   $C^2$-regular  with $P(E) \leq M$.  There exist a constant $C_M$ and  points $x_1, x_2, \dots, x_N$, with $|x_i -x_j| >2$,   such that  for $F = \cup_{i=1}^N D_1(x_i) $ it holds
\begin{equation}\label{alex-1}
\sup_{x\in E\Delta F}d_{\pa F}(x) \leq C_M \|k_E-1 \|_{L^1(\pa E)} 
\end{equation}
and 
\begin{equation}\label{alex0}
 | P(E) - 2\pi N| \leq C_M \|k_E-1 \|_{L^1(\pa E)} .
\end{equation}
Moreover, there exists $\eps_0>0$ such that if $ \|k_E-1 \|_{L^2(\pa E)} \leq \eps_0$ then $E$ is $C^1$-diffeomorphic to the disjoint union of $N$ disks. 
\end{lemma}

\begin{proof}
Assume  that $\|k_E-1 \|_{L^1(\pa E)}  \geq \eps_0$ for a small $\eps_0$ to be chosen later. Since $\|k_E\|_{L^1(\pa E)}<\infty$, $E$ has finitely many connected components $E_i$, $i=1,\dots,N$. If $P(E)\geq 2\pi N$, then $|P(E)-2\pi N|\leq M$, hence \eqref{alex0} follows with a sufficiently large constant. Otherwise, using Gauss-Bonnet theorem,  
$$
2\pi N-P(E)\leq\sum_{i=1}^N\int_{\pa E_i}(|k_E|-1)\,d\Ha^1\leq \|k_E-1\|_{L^1(\pa E)},
$$
hence \eqref{alex0} follows with $C_M=1$.
Since $P(E_i)\leq M$ for every $i$, there exist points $x_i$ such that $E_i\subset D_M(x_i)$.  Therefore $\sup_{x\in E_i\Delta D_1(x_i)}d_{\pa D_1(x_i)}(x)$  is smaller than $M$. Hence $\sup_{x\in E\Delta F}d_{\pa F}(x)\leq M$ and \eqref{alex-1} holds with a sufficiently large constant.  

Assume now that   $\|k_E-1 \|_{L^1(\pa E)}  \leq \eps_0$ for a small $\eps_0$. Let us fix a component $E_i$ of $E$ and denote $l = P(E_i)$. Let us first prove that there is ${x}_i$ such that  
\beq
\label{alex 1}
\sup_{x\in E\Delta D_1(x_i)}d_{\pa D_1(x_i)}(x) \leq C \|k_E-1 \|_{L^1(\pa E)} \qquad \text{and} \qquad |l - 2 \pi  | \leq   \|k_E-1 \|_{L^1(\pa E_i)} .
\eeq 
It is not difficult to see that the  claim follows from  \eqref{alex 1}. 

We claim  first that $E_i$ is simply connected. Indeed, let $\Gamma_0$ be the outer  component of $\pa E_i$ for which it holds $\int_{\Gamma_0} k_E \, d \Ha^1 = 2 \pi$.  Then it follows from  $\|k_E-1 \|_{L^1(\pa E)}  \leq \eps_0$ that  
\[
2 \pi - \Ha^1(\Gamma_0)  = \int_{\Gamma_0} (k_E -1) \, d \Ha^1 \leq \int_{\pa E} |k_E-1| \,  d \Ha^1 \leq \eps_0.
\]
This yields $P(E_i) \geq \Ha^1(\Gamma_0)  \geq 2 \pi - \eps_0$. Then 
\[
\int_{\pa E_i} k_E \, d \Ha^1 = \int_{\pa E_i} (k_E -1) \, d \Ha^1 + P(E_i) \geq P(E_i) - \int_{\pa E} |k_E-1| \,  d \Ha^1 \geq 2 \pi -  2\eps_0.
\]
Therefore when $\eps_0 < \pi $   we conclude that $\int_{\pa E_i} k_E \, d \Ha^1 $ is positive. Since $E_i$ is connected, this implies that   it is  simply connected. 

Since the boundary $\pa E_i$ is connected we may parametrize it by unit speed curve  $\gamma : [0, l] \to \R^2$,  $\gamma(s) = (x(s),y(s))$ with counterclockwise orientation. Define $\theta(s) := \int_{0}^s k_E(\gamma(\tau)) \, d \tau $ so that 
$\theta(0) = 0$ and $\theta(l) = 2 \pi$.  Then 
\beq
\label{alex 2}
|\theta(s) - s| \leq  \|k_E-1 \|_{L^1(\pa E_i)} \qquad \text{for all }\, s \in [0,l].
\eeq 
In particular, for $s = l$ \eqref{alex 2} implies  
\[
|\theta(l) - l|= |2 \pi - l| \leq     \|k_E-1 \|_{L^1(\pa E_i)}
\]
which is the second inequality in \eqref{alex 1}. 

 By possibly rotating the set $E$ we have 
$$
x'(s) = - \sin \theta(s) \qquad \text{and} \qquad y'(s) = \cos \theta(s) .
$$
In particular, \eqref{alex 2} implies
$$
|x'(s)  +  \sin  s| \leq  \|k_E-1 \|_{L^1(\pa E_i)}  \qquad \text{and} \qquad |y'(s)  -  \cos  s| \leq  \|k_E-1 \|_{L^1(\pa E_i)} 
$$
for all $s \in [0,l]$. Therefore there are numbers $a$ and $b$ such that 
\beq
\label{alex 5}
|x(s) - a  -  \cos  s| \leq C \|k_E-1 \|_{L^1(\pa E_i)}  \qquad \text{and} \qquad |y(s) - b  -  \sin  s| \leq C \|k_E-1 \|_{L^1(\pa E_i)} 
\eeq
for all $s \in [0,l]$. Therefore we obtain from $|l - 2 \pi  | \leq   \|k_E-1 \|_{L^1(\pa E)}$ that 
\[
|x(s) - a  -  \cos (2 \pi s /l )| \leq C \|k_E-1 \|_{L^1(\pa E_i)}  \qquad \text{and} \qquad |y(s) - b  -  \sin(2 \pi s /l )| \leq C \|k_E-1 \|_{L^1(\pa E_i)} ,
\]
which gives the first inequality in \eqref{alex 1} for ${x}_i = (a,b)$.

Note that from \eqref{alex 5} it follows that if  $\|k_E - 1 \|_{L^2(\pa E)}$ is small, then $\gamma(s)$ is close in $C^{1,\alpha}(0,l)$ to the parametrization $(a+\cos (2 \pi s /l ), b+\sin(2 \pi s /l ))$ of $\pa D_1( x_i)$. Hence   $E_i$ is $C^{1,\alpha}$-close  to $D_1({x}_i)$.

\end{proof}

The following lemma is based on a comparison argument. 
\begin{lemma}
\label{compa}
Assume $E_0  \subset  \R^2$ is    $C^2$-regular set with $P(E_0) \leq M$ and let $(E_t^h)_t$ be the  approximate flat flow  starting from $E_0$.   If  $E_0$ is close to a  disjoint union of $N$ disks with radius one, i.e., there exist  $F = \cup_{i=1}^N D_1(x_i)$, with $|x_i - x_j| \geq 2$ for $i \neq j$, such that 
\[
\sup_{x \in E^h_t \Delta  F}d_{\pa F}(x) \leq \delta,
\]
then  for $\delta>0$ small enough it holds
\[
\sup_{x \in E^h_t\Delta  F} d_{\pa F}(x) \leq  5\delta^{1/4} \qquad \text{for all }\, t \in (0,\sqrt{\delta})
\]
for all  $h >0$ small. 
\end{lemma}

\begin{proof}
Let $F$ be the union of disks as in the assumption and define 
\[
F_- := \{ x \in F :   d_{\R^2 \setminus F}(x) > \delta^{1/4}\} \qquad \text{and} \qquad F_+ := \{ x \in \R^2 : d_F(x) < \sqrt{\delta} \} . 
\]
Then clearly $F_- \subset F \subset F_+$ and by the assumption  $\sup_{x \in E_0 \Delta  F} d_{\pa F}(x) \leq \delta$ it  holds $F_- \subset E_0 \subset F_+$. 

Let  $(F_t^h)_t$ be the approximate flat flow with the constant forcing term $f = -\Lambda$,  where $\Lambda :=  C_0 +1$, with $C_0$  as in \eqref{f1}, starting from $F_-$. Then by Proposition \ref{compa prin}  it holds $F_t^h\subset E_t^h$ for all 
$t>0$. Note that $F_-$ is a union of  disks with radius $R = 1 -\delta^{1/4}$ and with positive distance to each other. It is easy to see that $(F_t^h)_t$ is decreasing, i.e., $F_t^h \subset F_s^h$ for $t >s$ and therefore 
it is enough to study the evolution of a one single disk $D_{R}$, because the flow $(F_t^h)_t$  is the union of them. If now $(\tilde{F}_t^h)$ is the  approximate flat flow starting from   $D_{R}$ with the forcing term $f = -\Lambda $ then it is not difficult to see that 
for $t \in (kh, (k+1)h]$ the set $\tilde{F}_t^h$ is a concentric disk with radius $r_{k+1}$ and by the Euler-Lagrange equation \eqref{euler} it holds
\[
\frac{ r_{k+1}  - r_k}{h}  = -\frac{1}{r_{k+1}}-\Lambda. 
\]
Therefore, it holds
\[
 r_{k+1}  - r_k  \geq -(\Lambda +2) h.
\]
for all  $k = 0,1,2,\dots $ for which $r_{k+1} \geq 1/2$. By adding this over  $k=0,1,\dots, K$  with  $\sqrt{\delta}/h \leq  K  \leq 2\sqrt{\delta}/h$ and recalling that  $r_0 = R =  1 -\delta^{1/4}$ we obtain
\[
r_K  \geq r_0  -2\sqrt{\delta}(\Lambda +2) \geq 1 -  2\delta^{1/4},
\]
when $\delta$ is small. This implies  $\sup_{x \in D_R \setminus  \tilde{F}_t^h} d_{\pa D_R}(x) \leq 2 \delta^{1/4}$ for $t \in (0, \sqrt{\delta})$ and thus by the previous discussion  
\[
\sup_{x \in F \setminus F_t^h}d_{\pa F}(x) \leq  2\delta^{1/4} \qquad \text{for } \,  t \in (0, \sqrt{\delta}).
\]
 Since $F_t^h\subset E_t^h$ we have 
\beq 
\label{compa1}
 \sup_{x \in F \setminus E_t^h}d_{\pa F}(x)  \leq  2 \delta^{1/4} \qquad \text{for } \,  t \in (0, \sqrt{\delta}). 
\eeq

We need yet to show that 
\beq
\label{compa2}
 \sup_{x \in  E_t^h \setminus F}d_{\pa F}(x)  \leq  5 \delta^{1/4} \qquad \text{for } \,  t \in (0, \sqrt{\delta}).
\eeq
 Denote $\Gamma =  \{x \in \R^2 \setminus F :d_{\pa F}(x) = 5\,  \delta^{1/4} \} $. Fix  $x \in \Gamma$  and denote the disk $D_r(x)$ with $r = 4 \, \delta^{1/4}$. Then by $E_0 \subset F_+$  and $D_r(x) \subset \R^2 \setminus F_+$ we have   $E_0 \subset   \R^2 \setminus D_r(x)$ if $\delta$ is sufficiently small. Let $(G_t^h)_t$ be the  approximate flat flow starting from   $D_r(x)$ with the constant forcing term $f = -\Lambda $.   Arguing as above we deduce that for $t \in (kh, (k+1)h]$ the set $G_t^h$ is  disk with radius $r_{k+1}$, i.e.,  $G_t^h = D_{r_{k+1}}(x)$ and 
\[
\frac{ r_{k+1}  - r_k}{h}  = -\frac{1}{r_{k+1}}-\Lambda \geq - \delta^{-1/4} - \Lambda
\]
 for $k = 0,1,\dots$ for which $r_{k+1} \geq \delta^{1/4}$. By adding this  over  $k=0,1,\dots, K$  with  $\sqrt{\delta}/h \leq  K  \leq 2\sqrt{\delta}/h$ and recalling that $r_0 =r = 4\, \delta^{1/4}$ we obtain 
\[
r_K  \geq r_0  - 2 \sqrt{\delta} (\delta^{-1/4} + \Lambda) \geq 4 \,  \delta^{1/4}  - 3  \, \delta^{1/4} = \, \delta^{1/4},
\]
when $\delta$ is small. In other words  $D_{\delta^{1/4}}(x) \subset  G_t^h$ for all $t \in (0, \sqrt{\delta})$. Since $E_0 \subset   \R^2 \setminus D_r(x)$  Proposition \ref{compa prin} yields 
\[
 E_t^h \subset \R^2 \setminus  G_t^h \subset \R^2 \setminus D_{ \delta^{1/4}}(x)
 \]
for all $0<t \leq \sqrt{\delta}$.  By repeating the same argument for all  $x \in \Gamma$  we conclude that the flow $E_t^h$ does not intersect $\Gamma$ for any $t \in (0,\sqrt{\delta})$. This implies  \eqref{compa2}. 
\end{proof}

In the next lemma we show that if $E_0$ is stationary then necessarily it is a disjoint union of disks, i.e., a critical set of the isoperimetric problem. 

\begin{lemma}
\label{stat-crit}
Assume  $E_0 \subset \R^2$ is a bounded  set of finite perimeter. If $E_0$ is stationary according to Definition \ref{stationary} then it is a disjoint union of disks with equal radii.
\end{lemma} 

\begin{proof}
Let us fix $T>1$ and $\eps>0$ and  let $(E_t^h)_t$ be  an approximate flat flow starting from $E_0$. Then for any $\delta >0$  it holds by  Definition \ref{stationary} and by Proposition \ref{MSS}  that 
\beq \label{s-c 1}
\sup_{h< t \leq T} |E_t^h \Delta E_0| \leq \delta 
\eeq
for small $h$. Now the forcing term satisfies trivially the assumption \eqref{f2} and therefore the left hand side of  \eqref{prop MSS2 1} is always zero.   Then from the proof of Proposition \ref{MSS2} (i) we get that for every $h$ sufficiently small 
\[
c \int_{2h}^{T}\int_{\pa E_t^h} (k_{E_t^h} - c_0)^2 \, d\Ha^{1}dt  + \mathcal{E}(E_{T}^h )\leq  \mathcal{E}(E_{h}^h).
\]
Recall that $ \mathcal{E}(E) = P(E) - c_0 |E|$. By \eqref{s-c 1} it holds $|E_T^h \Delta E_h^h| \leq 2\delta < \frac{\eps}{c_0}$.  Therefore we have
\[
c \int_{2h}^{T}\int_{\pa E_t^h} (k_{E_t^h} - c_0)^2 \, d\Ha^{1}dt + P(E_T^h) \leq P(E_h^h) + \eps. 
\]
Finally by \eqref{prop MSS 3} and \eqref{s-c 1} we obtain
\[
P(E_h^h) \leq P(E_0) + c_0( |E_{h}^h| -  |E_0| ) \leq  P(E_0) + c_0 \delta \leq  P(E_0) + \eps. 
\]
Hence,
\[
c \int_{2h}^{T}\int_{\pa E_t^h} (k_{E_t^h} - c_0)^2 \, d\Ha^{1}dt + P(E_T^h) \leq P(E_0) + 2\eps. 
\]
By \eqref{s-c 1} it holds $|E_T^h \Delta E_0| \leq \delta$. Therefore by the lower semicontinuity of the perimeter it holds $P(E_0) \leq P(E_T^h) + \eps$ when $\delta$ and $h$ are small. Therefore we have
\[
c \int_{2h}^{T}\int_{\pa E_t^h} (k_{E_t^h} - c_0)^2 \, d\Ha^{1}dt   \leq 3\eps.
\]
By the mean value theorem there is $t <T$ such that 
\[
\| k_{E_t^h} - c_0\|_{L^2} \leq C \sqrt{\eps}. 
\] 
Since by Proposition~\ref{MSS2} $\sup_{t\geq0}P(E^h_t)\leq M$ for some $M$ independent of $h$,  from the previous inequality and from  Lemma \ref{alex} it follows that there are points $x_1, \dots, x_N$, with $|x_i - x_j| \geq 2r $, where $r = \frac{1}{c_0}$, such that for the set $F= \cup_{i=1}^N D_r(x_i)$ it holds
\[
\sup_{x \in E_t^h \Delta F} d_{\pa F}(x) \leq C \sqrt{\eps}.
\]
Thus by \eqref{s-c 1} it holds 
\[
|E_0 \Delta F| \leq C \sqrt{\eps}. 
\]
Note that the points $x_i$ might depend on $t$ and  on $h$ but the radius $r$ and the  number of disks $N$ does not. Therefore we  conclude that the set $E_0$ is arbitrarily close to a  union  of essentially disjoint   disks.
This implies that the set $E_0$  itself is a  union of essentially disjoint disks with radii $r = \frac{1}{c_0}$.  
\end{proof}

\begin{proof}[\textbf{Proof of  Theorem \ref{thm1}}]

Without loss of generality we may assume that 
\[
E_0 = D_1(-e_1) \cup D_1(e_1)
\]
where $e_1 = (1,0)$. Let us now fix a small $h>0$  and consider the minimization problem \eqref{min mov} which gives a
  sequence of sets $(E^{h,k})_{k=1}^\infty$ and thus an approximate flat flow $(E_t^h)_t$. 

 Let us fix $\eps_0>0$. Then for  $\delta$ small enough  we have  by Lemma  \ref{compa}   that  for $k \leq \frac{\delta}{h}$ it holds 
\beq
\label{balls inside}
\left( D_{1-\eps_0}(-e_1) \cup D_{1-\eps_0}(e_1)\right) \subset E^{h,k}  \subset  \left(D_{1+\eps_0}(-e_1) \cup D_{1+\eps_0}(e_1)\right),
\eeq
when $h$ is small. Moreover,  by Lemma \ref{distance} it holds 
\[
\left(D_{1-C_1\sqrt{h}}(-e_1) \cup D_{1-C_1\sqrt{h}}(e_1) \right) \subset E^{h,1}.
\]

Let us improve the above estimate and show that the set $ E^{h,1}$ contains a large simply connected set.  To be more precise, we denote  the rectangle $R^h_{\eta} = (-2C_1 \sqrt{h}, 2C_1 \sqrt{h}) \times  (-\eta h^{1/4},\eta h^{1/4})$  and prove  that for $\eta>0$ small, independent of $h$, it holds
\beq
\label{thm2 st2 1}
\left( D_{1-C_1\sqrt{h}}(-e_1) \cup D_{1-C_1\sqrt{h}}(e_1)\right)  \cup R^h_{\eta} \subset E^{h,1}.
\eeq
We argue by contradiction assuming that $\pa E^{h,1} \cap  (-2C_1 \sqrt{h}, 2C_1 \sqrt{h}) \times  (-\eta h^{1/4},\eta h^{1/4})$ is non-empty. We denote 
  the rectangle $R^h_{3 \eta}  =(-2C_1 \sqrt{h}, 2C_1 \sqrt{h}) \times  (-3\eta h^{1/4}, 3 \eta h^{1/4})$ and define  
\[
\tilde{E}^{h,1} =E^{h,1} \cup R^h_{3 \eta}.
\]
By the contradiction assumption the length of the curve  $\pa E^{h,1} \cap R^h_{3 \eta}$ is greater than $2 \eta  h^{1/4}$.   We have then the following  estimate 
\beq
\label{peri smaller}
P(\tilde{E}^{h,1}) \leq P(E^{h,1}) - \eta h^{1/4} ,
\eeq
when $h$ is small. We also have 
\begin{itemize}
\item[(i)] $\big| |\tilde{E}^{h,1}| - |E^{h,1}| \big| \leq |R^h_{3 \eta}| = 24 C_1  \eta h^{\frac34}$
\item[(ii)] $|\bar{d}_{E_0}(x)| \leq 10 \eta^2\sqrt{h}$ for all $x \in R^h_{3 \eta} \setminus E_0$. 
\end{itemize}
It follows from (i),   (ii)  and  $\sup_t |f(t)| \leq C_0$ that 
\[
\begin{split}
\frac{1}{h} \int_{\tilde{E}^{h,1}} &\bar{d}_{E_0} \, dx  + \bar f(h) |\tilde{E}^{h,1}| \\
&\leq  \frac{1}{h} \int_{E^{h,1}} \bar{d}_{E_0} \, dx  + \bar f(h) |E^{h,1}|  +  \frac{1}{h} \int_{R^h_{3\eta} \setminus E_0} \bar{d}_{E_0} \, dx + C_0 \big| |\tilde{E}^{h,1}| - |E^{h,1}| \big| \\
&\leq    \frac{1}{h} \int_{E^{h,1}} \bar{d}_{E_0} \, dx  + \bar f(h) |E^{h,1}|  + C \eta^3 h^{1/4} 
\end{split}
\]
when $h$ is small.  Therefore  using \eqref{peri smaller} and the above inequality we may estimate
\[
\begin{split}
\mathcal{F}(\tilde{E}^{h,1};E_0) &\leq \mathcal{F}(E^{h,1};E_0) -   \eta h^{1/4} + C \eta^3 h^{1/4}  <  \mathcal{F}(E^{h,1};E_0) 
\end{split}
\]
when $\eta>0$ is small enough. This contradicts the minimality of $E^{h,1}$ and we obtain \eqref{thm2 st2 1}.

\medskip

We continue  by constructing a barrier set $G_h$  (see Figure \ref{fig2}) and prove that $G_h \subset E^{h,k}$ for every $k \leq \delta/h$.  For $h\geq0$ we define $\varphi_h : (-3\eps_0, 3 \eps_0) \to \R$ as
\[
\varphi_h(s) = 3 \eps_0 - \sqrt{9 \eps_0^2  - s^2} + h
\]
and define the set
\[
G_{\varphi_h} := \{(x_1,x_2) \in \R^2  :  x_1 \in (-3\eps_0, 3 \eps_0) , \,\, |x_2| < \varphi_h(x_1) \} ,
\] 
which is 'the neck'. We define the barrier set as 
\[
G_h = \left( D_{1-2\eps_0}(-e_1) \cup D_{1-2\eps_0}(e_1)\right) \cup G_{\varphi_h}.
\]
 The barrier set  $G_h$ is open and connected and we have the estimate on the curvature at the neck
\beq
\label{thm2 st3 1}
\text{\emph{for every }}\, x \in \pa G_h \setminus  (\bar{D}_{1 -2\eps_0}(-e_1) \cup \bar{D}_{1 -2\eps_0}(e_1) )\,\, \text{\emph{it holds }}\, k_{G_0}(x) = -\frac{1}{3 \eps_0}.
\eeq 
Moreover, we notice that when $h$ is small then by \eqref{thm2 st2 1} it holds 
\[
G_h \subset E^{h,1} 
\]
In fact,   \eqref{thm2 st2 1} implies that 
\beq
\label{thm2 st3 2}
\inf_{\R^2 \setminus E^{h,1}} d_{G_h}(x) \geq c\, h^{1/4}
\eeq 
for  small $c>0$. 

\begin{figure}
\begin{tikzpicture}
\clip (-5, -3.3) rectangle (5, 3.3);
\begin{scope}[yscale=1.8,xscale=1.8]

\fill[draw=black, thin, fill=black!20!white] (-1.2,0)+(1.2,0.08)  .. controls (-.05, .08) and (-.2, .12) ..  +(15:1) arc (15: 345: 1)++(0,0) .. controls (-.2, -.12) and (-.05, -.08) .. (0,-0.08)--(0,0.08);

\begin{scope}[yscale=1,xscale=-1]
\fill[draw=black, thin, fill=black!20!white] (-1.2,0)+(1.2,0.08)  .. controls (-.05, .08) and (-.2, .12) ..  +(15:1) arc (15: 345: 1)++(0,0) .. controls (-.2, -.12) and (-.05, -.08) .. (0,-0.08)--(0,0.08);
\end{scope}

\draw (0, 0.4) 
.. controls ++(180:0.1) 
and ++(330:0.3) 
.. (-0.8, 1.1)  
.. controls ++(330:-0.3) and ++(40:0.3) .. (-1.7, 1.1) 
.. controls ++(40:-0.3) and ++(90:.5) .. (-2.4, 0);

\begin{scope}[yscale=1,xscale=-1]
\draw (0, 0.4) .. controls ++(180:0.1)  and ++(330:0.3) .. (-0.8, 1.1)  
.. controls ++(330:-0.3) and ++(40:0.3) .. (-1.7, 1.1)
.. controls ++(40:-0.3) and ++(90:.5) .. (-2.4, 0);
\end{scope}

\begin{scope}[yscale=-1,xscale=1]
\draw (0, 0.4) .. controls ++(180:0.1)  and ++(330:0.3) .. (-0.8, 1.1)  
.. controls ++(330:-0.3) and ++(40:0.3) .. (-1.7, 1.1)
.. controls ++(40:-0.3) and ++(90:.5) .. (-2.4, 0);
\end{scope}

\begin{scope}[yscale=-1,xscale=-1]
\draw (0, 0.4) .. controls ++(180:0.1)  and ++(330:0.3) .. (-0.8, 1.1)  
.. controls ++(330:-0.3) and ++(40:0.3) .. (-1.7, 1.1)
.. controls ++(40:-0.3) and ++(90:.5) .. (-2.4, 0);
\end{scope}

\draw[->] (-2.8,0) -- (2.8,0); 
\draw (2.6,0.2) node 
                       {$x_1$};
\draw[->] (0,-1.8) -- (0,1.8); 
\draw (0.3,1.6) node 
                       {$x_2$};
\draw (1.5, 0.3) node 
                       {$G_h$};
\draw (-1.9, 1.2) node 
                       {$\partial E^{h,k}$}; 
\end{scope}                      
\end{tikzpicture}
\caption{The boundary of $\pa E^{h,k}$ lies outside of the barrier set $G_h$.}\label{fig2}
\end{figure}
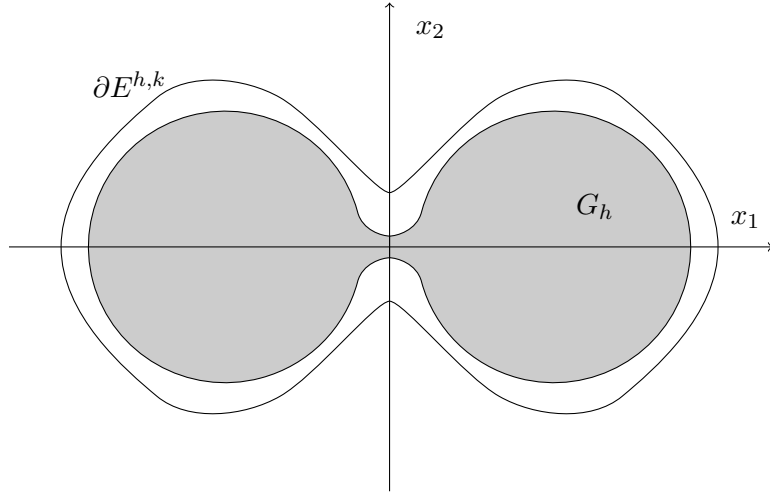

Let us define 
\[
\rho_k := \inf_{\R^2 \setminus E^{h,k}} d_{G_h}(x)
\]
for $k=1,2,\dots$ and $\rho_0 := 0$.  We claim  that   for every $ k \leq \frac{\delta}{h}$ it holds
\beq
\label{thm2 st3 3}
\rho_{k+1} - \rho_k \geq 2h \qquad \text{or} \qquad  \rho_{k+1} \geq \frac{\eps_0}{2}
\eeq 
when $h$ is small. 

We prove \eqref{thm2 st3 3} by induction and notice that for $k=0$ the inequality \eqref{thm2 st3 3} is already proven since \eqref{thm2 st3 2} implies  $\rho_1 \geq c\, h^{1/4}$.  Let us assume that    \eqref{thm2 st3 3} holds for $k-1$ and prove it for $k$.  Let us  assume that  $\rho_{k+1} < \frac{\eps_0}{2}$.  The induction assumption and  $\rho_1 \geq c\, h^{1/4}$ yields $\rho_{k} \geq c\, h^{1/4}$. On the other hand by Lemma \ref{distance} it holds $\sup_{E^{h,k+1} \Delta E^{h,k}} d_{\pa E^{h,k}} \leq C_1\sqrt{h}$ and therefore $\rho_{k+1} >0$. 

 Let $x_{k+1} \in  \pa E^{h,k+1}$ and $y_{k+1} \in \pa G_h$ be such that $|x_{k+1}-y_{k+1}| =  \min_{x \in \pa E^{h,k+1}} d_{G_h}(x) = \rho_{k+1}$. By  \eqref{balls inside} it holds $x_{k+1} \notin D_{1-\eps_0}(-e_1) \cup D_{1-\eps_0}(e_1)$ and therefore by $\rho_{k+1} < \frac{\eps_0}{2}$ we have
\[
y_{k+1} \in  \pa G_h \setminus  (\bar{D}_{1 -2\eps_0}(-e_1) \cup \bar{D}_{1 +2\eps_0}(e_1) ).
\]
Then \eqref{thm2 st3 1} yields
\[
k_{G_h}(y_{k+1}) = -\frac{1}{3 \eps_0}.
\]
Since $x_{k+1}$ is a point of minimal distance  $k_{E^{h,k+1}}(x_{k+1}) \leq k_{G_h}(y_{k+1}) = -\frac{1}{3 \eps_0}$. By taking $\eps_0$ smaller, if needed,  we have by the Euler-Lagrange equation \eqref{euler} and by $\sup_{t>0}|f(t)| \leq C_0$ that
\beq \label{thm2 st3 a}
\frac{ \bar{d}_{E^{h,k}}(x_{k+1})}{h} = -k_{E^{h,k+1}}(x_{k+1}) + \bar f(kh) \geq  \frac{1}{3 \eps_0} -C_0\geq  2.
\eeq

The inequality \eqref{thm2 st3 a} and  $G_h \subset E^{h,k}$ imply that $x_{k+1} \notin E^{h,k}$ and $y_{k+1} \in E^{h,k}$. Thus there is a point  $z_{k+1}$ on the segment $[x_{k+1},y_{k+1}]$ such that $z_{k+1} \in \pa E^{h,k}$. Since $y_{k+1} \in  \pa G_h$ it holds  $|z_{k+1}-y_{k+1}| \geq \rho_{k}$ and  by \eqref{thm2 st3 a}
 we have $|x_{k+1}-z_{k+1}|\geq \bar{d}_{E^{h,k}}(x_{k+1}) \geq 2h$. Therefore because  $z_{k+1}$ is on the segment $[x_{k+1},y_{k+1}]$ we have
\[
\rho_{k+1} = |x_{k+1}-y_{k+1}|  = |x_{k+1}-z_{k+1}|  + |z_{k+1}-y_{k+1}|  \geq 2h + \rho_{k}.
\]
Thus we have  \eqref{thm2 st3 3}.

Let us conclude the proof. By adding  \eqref{thm2 st3 3} together for  $k=0,1,2, \dots, K$ with $ K \leq \delta/h$ we deduce that  for $\delta$ small it holds 
\beq \label{thm2 st3 b}
\inf_{\R^2 \setminus E_t^h} d_{G_h}(x) \geq  t \qquad \text{for all }\, t \in (h,\delta].  
\eeq
In particular $G_h \subset E_t^h$. Let $h_l\to 0$ be any sequence such that $\sup_{h_l< t\leq\delta}|E^{h_l}_t\Delta E_t|\to0$, see Proposition \ref{MSS}. By  \eqref{thm2 st3 b} we get
\[
\inf_{\R^2 \setminus E_t} d_{G_0}(x) \geq  t \qquad \text{for all }\, t \in (0,\delta].
\]
This inequality implies, in particular, that $E_t$ contains a ball centered at the origin with radius $t$ for all $t \in(0,\delta)$  and that 
\[
 |E_t \setminus E_0| \geq c\,t^3.
\]
This is the second statement  of Theorem \ref{thm1}.
The inequality \eqref{thm2 st3 b} implies that 
\[
 \{ x \in \R^2 : d_{G_h}(x) <t  \} \subset E_t^h  \qquad \text{for all }\, t \in (h,\delta].
\]
Passing to the limit as above, along the subsequence $h_l$, we deduce 
\[
 \{ x \in \R^2 : d_{G_0}(x) <t  \} \subset E_t  \qquad \text{for all }\, t \in (0,\delta].
\]
The first claim follows from the fact that $ \{ x \in \R^2 : d_{G_0}(x) <t  \}$ is open and  simply connected.
\end{proof}

We conclude this section by explaining how Corollary \ref{corollary} follows from Theorem \ref{thm1}.
\begin{proof}[\textbf{Proof of Theorem \ref{corollary}}]
First, it is easy to see that if $E_0$ is a  union of  disks with equal radius and with positive distance to each other, then $E_0$ is stationary according to the Definition~\ref{stationary}. If $E_0$ is stationary then by Lemma \ref{stat-crit} it is critical, i.e., finite union of 
essentially disjoint disks $D_r(x_i)$, with $r=1/c_0$ and $i=1,\dots,N$.  We need to show if $i\not=j$ then $|x_i-x_j|>2/c_0$.  If by contradiction there are two tangential disks, say $D_r(x_1)$ and $D_r(x_2)$, then we define $F_0 = D_r(x_1) \cup D_r(x_2)$. Let $(E_t)_t$ be a flat flow starting from $E_0$ and let $h_l\to0$ be a sequence such that   $|E^{h_l}_t\Delta E_t|\to0$ and  $|F^{h_l}_t\Delta F_t|\to0$, where $(F_t)_t$ is a flat flow starting from $F_0$ with  forcing term $g = c_0 -\eps$. Then by Proposition \ref{compa prin}  $F_t ^{h_l}\subset E_t^{h_l}$ for all $t>0$ and $h_l$, hence $F_t \subset E_t$. By Theorem \ref{thm1} we have that there exist $\delta, c>0$ such that for all $t\in(0,\delta)$ 
\[
 |F_t \setminus F_0| \geq c\,t^3.
\]
This implies 
\[
 |E_t \setminus E_0|  \geq c\,t^3
\]
and therefore $E_0$ is not stationary. 
\end{proof}

\section{Proofs of  Theorem \ref{thm2} and Theorem \ref{thm3}}

\begin{proof}[\textbf{Proof of  Theorem \ref{thm2}}]

First, by scaling we may assume that $c_0 =1$ in the assumption \eqref{f2}.

Let $(E_t^h)$ be an approximate flat flow which converges, up to a subsequence, to $(E_t)_t$. We simplify the notation and denote the converging subsequence again by $h$.    From the assumption $\sup_{t >0} |E_t|= M$ and from  Proposition \ref{MSS} (iv)   it follows that  for every $T>0$ there is $h_{ T}$ such that up to subsequence of $h$ it holds 
\[
\sup_{0< t <T} |E_t^h| \leq 2M \qquad \text{for all } \, 0<h<h_{ T}.
\]
Then Remark~\ref{rem energy} yields that there exists a constant $\tilde C$ independent of $h$ and $T$ such that for $0<h<h_T$
\[
\sup_{0<t<T} P(E_t^h) \leq \tilde{C}.
\]

The dissipation inequality in Proposition \ref{MSS2} and the  above volume and perimeter  bounds imply 
\[
\int_{T_0}^{T} \int_{\pa E_t^h} (k_{E_t^h} - \bar f(t-h))^2 \, d \Ha^1 dt \leq C 
\]
for some $T_0>0$  for every $T > T_0$. Then the assumption \eqref{f2} (recall that $c_0 = 1$) yields
$$
\int_{T_0}^{T} \int_{\pa E_t^h} (k_{E_t^h} - 1)^2 \, d \Ha^1 dt \leq C .
$$
for some $T_0$ large and for every $T > T_0$ and $0<h<h_T$.  In particular, if we denote $I_j = [(j-1)^2, j^2]$ for  $j =1,2,\dots , k < \sqrt{T} $ then it  holds  
\[
\int_{I_j} \int_{\pa E_t^h} (k_{E_t^h} - 1)^2 \, \, d \Ha^1 dt \leq C
\]
for $j$ large. Let us fix a small $\eps>0$. From the previous inequality we obtain that there exists $j_\eps$ such that, if $j_\eps\leq j\leq\sqrt{T}$ and $0<h<h_T$ there exists $T_{h,j}$ such that 
\[
 (j-1)^2 \leq T_{h,j} \leq j^2,  \qquad  \text{and } \qquad   \Big\| k_{E_{T_{h,j}}^h} - 1\Big\|_{L^2(E_{T_{h,j}}^h)}  \leq \eps.
\]
We  deduce by Lemma \ref{alex} that the set  $E_{T_{h,j}}^h$ is close to  a disjoint union of $N_{h,j}$   disks of radius one. Since the measures of $E_{T_{h,j}}^h$ are 
uniformly bounded, we conclude that there is $N_0$ such that $N_{h,j} \leq N_0$.  Moreover, we have by Lemma \ref{alex} that 
\beq \label{thm1 2}
\big| P(E_{T_{h,j}}^h) - 2 \pi N_{h,j}\big| + \big| |E_{T_{h,j}}^h| -  \pi N_{h,j}\big|  \leq C \eps.
\eeq
This implies the following estimate for the energy $\mathcal{E}(E_{T_{h,j}}^h) = P(E_{T_{h,j}}^h) - |E_{T_{h,j}}^h| $,
\beq\label{thm3bis}
\big|  \mathcal{E}_{T_{h,j}}(E_{T_{h,j}}^h)  - \pi N_{h,j} \big| \leq   C \eps.
\eeq
In other words, at $T_{h,j}$ the energy has almost the value $\pi N_{h,j}$. Since the energy $\mathcal{E}_{T_{h,j}}(E_{T_{h,j}}^h)$ is asymptotically  almost decreasing by Remark \ref{rem energy}, we deduce that the sequence of numbers $N_{h,j}$ is decreasing for $j$ large, i.e., 
\[
N_{h,j} \geq N_{h,j+1} \qquad \text{for every }\, j_\eps \leq j \leq \sqrt{T}.
\]

By letting $h \to 0$ we conclude by  Proposition \ref{MSS} (iv) and by standard diagonal argument that, by extracting another subsequence if needed,  there is sequence of times $T_j $, with $j \geq j_\eps$,  such that $ (j-1)^2 \leq T_j\leq j^2$ and the set  $E_{T_j}$ is close to $N_{j}$ many disjoint disks of radius one and that $N_j \geq N_{j+1}$ for every $j \geq j_\eps$.  This implies that, there is $j_0 \geq j_\eps$ and $N$ such that 
\[
N_{j} = N \qquad \text{for all } \,   j \geq j_0.
\]
 This means that every $E_{T_{j}}$, for $j\geq j_0$, is close in $L^1$-sense to disjoint union of exactly  $N$ many disks of radius one. By the locally uniform $L^1$-convergence $(E_t^h)_t \to (E_t)$  for every $T>j_0^2$ we have 
 \beq\label{thm3ter}
 N_{h,j} = N \qquad \text{for $j_0 \leq j \leq \sqrt{T},$}
 \eeq
 when $h$ is small. Therefore we conclude from \eqref{thm1 2} and  from the dissipation inequality  in Proposition \ref{MSS2} that  for any  $\delta>0$  there is $T_\delta$ such that for all $T > T_\delta$ it holds
\beq  \label{thm1 3}
\int_{T_\delta}^T \int_{\pa E_t^h} (k_{E_t^h} - 1)^2 \, \, d \Ha^1 dt \leq \delta^3
\eeq
 when $h$ is small.

Let us fix $T >> T_\delta$ and denote by $J_h \subset  (T_\delta,T)$ the set of times $t \in (T_\delta,T)$ for which 
\[
\|k_{E_t^h} - 1\|_{L^2(\pa E_t^h)} \geq \delta.
\]
Then by \eqref{thm1 3} it holds $|J_h| \leq \delta$.   If $\delta$ is small enough, by Lemma \ref{alex},  from \eqref{thm3bis}, \eqref{thm3ter} and \eqref{remenergy1} we deduce that the sets $ E_t^h$ satisfy 
\beq \label{thm1 5}
\sup_{x \in E_{t}^h \Delta F_{t}^h} d_{\pa F_{t}^h}(x) \leq C\delta \qquad \text{for  all } \, t \in (T_\delta,T) \setminus J_h,
\eeq
where $F_t^h = \cup_{i=1}^N D_1(x_i)$ with $|x_i-x_j| \geq2 $ for $i \neq j$. Note that the points $x_i$ may depend on $t$ and $h$.

We will  show that for all  $t \in (T_\delta +2\delta,T) $  it holds 
\beq \label{thm1 4}
\sup_{x \in E_{t}^h \Delta F_{t}^h} d_{\pa F_{t}^h}(x) \leq C \delta^{1/4}
\eeq
where $F_t^h$ is a union of $N$ disjoint disks as above.  Let us fix   $t_0\in (T_\delta, T) \setminus J_h$. By \eqref{thm1 5} we have    
\[
\sup_{x \in E_{t_0}^h \Delta F_{t_0}^h}d_{\pa F_{t_0}^h}(x) \leq C \delta.
\]
We use  Lemma \ref{compa} with $E_0 = E_{t_0}^h$ to conclude 
\[
\sup_{x \in  E_{t}^h \Delta F_{t_0}^h}d_{\pa F_{t_0}^h}(x) \leq C \delta^{1/4} \qquad \text{for all } \, t \in [t_0, t_0 +  \sqrt{\delta}).
\]
This means that if we define 
\[
I = \cup_{t \in (T_\delta,T) \setminus J_h} [t, t + \sqrt{\delta})
\]
then \eqref{thm1 4} holds for every  $t \in I$.   But since $|J_h| \leq \delta \leq \sqrt{\delta}/2$,   it easy to see that  $(T_\delta+2\delta,T) \subset I$. Thus  \eqref{thm1 4} holds for every $t \in (T_\delta +2\delta,T) $. 

We have thus proved \eqref{thm1 4}. The claim follows by letting $h \to 0$ and from Proposition~\ref{MSS}. 
\end{proof}

We conclude the paper by proving Theorem~\ref{thm3}. To this end we  recall the Bonnesen symmetrization of a planar set.

Let $E \subset \R^2$ be a measurable  set.  The Bonnesen symmetrization  of $E$ with respect to $x_2$-axis is the set $E^*$, with  the property that for every $r>0$ 
$$
\Ha^1(\pa D_r\cap E^*)=\Ha^1(\pa D_r\cap E)
$$
and  $\pa D_r \cap E^*$ is the union of two circular arcs $\gamma^+_r$ and $\gamma^-_r$ with equal length, symmetric with respect to the  $x_2$-axis and such that $\gamma^-_r$ is obtained by reflecting $\gamma^+_r$ with respect to the $x_1$-axis.

Clearly this symmetrization leaves the area unchanged. Moreover,  if $E$ is a convex set, symmetric with respect to both coordinate axes, then $P(E^*) \leq P(E)$, see  \cite[Page 67]{Bon} (see also \cite{CiLe}).

Let us  prove that the Bonnesen symmetrization decreases the dissipation.   

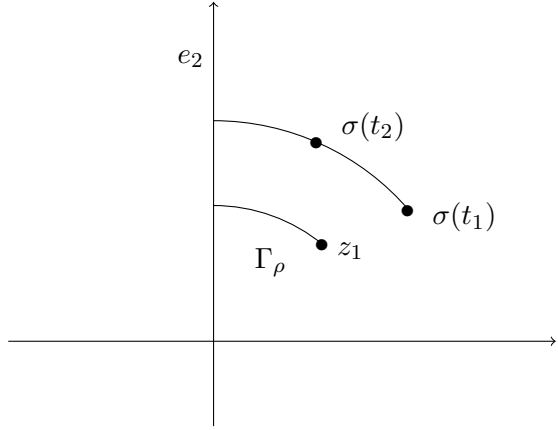
\begin{figure} \label{bonne2}
\begin{tikzpicture}
\begin{scope}[yscale=1.5,xscale=1.5] 
\clip (-3, -3) rectangle (3, 3);


\begin{scope}[yscale=1.5,xscale=1.5] 
\draw  (0,-1.2)+(50:1) arc (50: 90: 1);

\draw  (0,-1.2)+(40:1.5) arc (40: 90: 1.5);

\draw[->] (-1.2,-1) -- (2,-1); 

\draw[->] (0,-1.5) -- (0,1); 
\end{scope}

\draw (0.5,-0.8) node 
                       {$\Gamma_\rho$};

\draw (2.2,-0.4) node 
                       {$\sigma(t_1)$};

\draw (1.4,0.4) node 
                       {$\sigma(t_2)$};

\draw (1.7,-0.35) node 
                       {$\bullet$};

\draw (0.9, 0.25) node 
                       {$\bullet$};

\draw (0.95, -0.65) node 
                       {$\bullet$};

\draw (1.2, -0.7) node 
                       {$z_1$};

\draw (-0.2,1) node 
                       {$e_2$};

\end{scope}[yscale=1.5,xscale=1.5] 
\end{tikzpicture}
\caption{The point $\sigma(t_2)$ is closer to the arc $\Gamma_\rho$ than $\sigma(t_1)$.}\label{bonne2}

\end{figure}

\begin{lemma}
\label{bonnesen}
Let $G \subset \R^2$ be  invariant under Bonnesen symmetrization. Then for any  measurable set $E \subset \R^2$ it holds
\[
 \int_{E^*} \bar{d}_{G} \, dx \leq  \int_E \bar{d}_{G} \, dx .
\]
\end{lemma}
\begin{proof}
It is enough to prove that for every $r >0$ it holds  
\beq \label{bonne1}
\int_{\pa D_r \cap E^*}  \bar{d}_{G} \, d \Ha^1 \leq \int_{\pa D_r \cap E}  \bar{d}_{G} \, d \Ha^1.
\eeq
Let us fix $r>0$ and without loss of generality we may assume that $r=1$. Let $\sigma :[-\pi,\pi] \to \R^2$, 
\[
\sigma(t) = \begin{bmatrix}
		\cos(t) \\
		\sin(t)
	\end{bmatrix}.
\]
Since $G$ is symmetric with respect to both coordinate axes, the function  $t \mapsto \bar{d}_{G}(\sigma(t)) $ is  even and for every $ t \in (0,\pi/2)$  it holds $\bar{d}_{G}(\sigma(\pi - t)) = \bar{d}_{G}(\sigma(t))$. We observe that   \eqref{bonne1} follows once we show that 
$$
t \mapsto \bar{d}_{G}(\sigma(t)) \qquad \text{is decreasing on }\, t \in (0,\pi/2). 
$$

To this aim we fix $0 < t_1 < t_2 < \pi/2$. Let us assume that $\sigma(t_1) \in \R^2 \setminus G$,  the case $\sigma(t_1) \in G$ being similar. If $\sigma(t_2) \in  G$ then trivially $\bar{d}_{G}(\sigma(t_2)) \leq 0 \leq \bar{d}_{G}(\sigma(t_1))$.  Let us thus assume that $\sigma(t_2) \in \R^2 \setminus G$. Let $z_1 \in \pa G$ be such that $\bar{d}_{G}(\sigma(t_1)) = |\sigma(t_1) -z_1|$ and let $\rho >0 $ and $\theta_1 \in (0, \pi/2)$ be such that $z_1 = \rho \sigma(\theta_1)$. We denote by $\Gamma_\rho \subset \pa D_\rho$  the arc with endpoints $\rho e_2$ and $z_1$, i.e., 
\[
\Gamma_\rho = \{ \rho \sigma(t) : t \in (\theta_1, \pi/2)\}.
\]
Since $G$ is invariant under Bonnesen symmetrization we have $\Gamma_\rho \subset  G$. But now since $t_1 < t_2 < \pi/2$ it clearly holds (see Figure \ref{bonne2})
\[
\text{dist}(\sigma(t_2), \Gamma_\rho) \leq   \text{dist}(\sigma(t_1), \Gamma_\rho) = |\sigma(t_1) - z_1| = \bar{d}_{G}(\sigma(t_1)) .
\]
Since $\Gamma_\rho \subset G$ we have 
\[
\bar{d}_{G}(\sigma(t_2)) \leq \text{dist}(\sigma(t_2), \Gamma_\rho)
\]
and the claim follows. 

%

\end{proof}

\bigskip

\begin{proof}[\textbf{Proof of  Theorem \ref{thm3}}]
Let $G$ be the ellipse  
\[
G = \{ (x_1,x_2) \in \R^2 : \,  a^2 x_1^2 + x_2^2 < 1 \} \qquad \text{with } \, a>1
\]
as in the assumption and let $(G_t)_t$ be the classical solution of  the volume preserving mean curvature flow
\beq \label{thm3 1}
V_t = - k_{G_t} + \bar{k}_{G_t}
\eeq
starting from $G$.  By \cite{Hu},  $(G_t)_t$ is well defined for all times,  remains smooth and uniformly convex and converges exponentially fast to the disk 
$D_\rho$, where $\rho = \frac{1}{\sqrt{a}}$. Moreover $G_t\not= D_\rho$ for all $t>0$. Let us define $f(t) := \bar{k}_{G_t}$ which is therefore a smooth function  and converges exponentially fast to $1/\rho$. Note that then  $f$ satisfies \eqref{f2} for $c_0 = 1/\rho$.

By the regularity of $E_0$ the flat flow $(E_t)_t$ with the forcing term $f$ starting from $E_0$ coincides with the unique classical solution provided by \cite{Hu}, see \cite[Proposition 4.9]{CN}. Therefore, by the symmetry of $E_0$ we may conclude that
\beq \label{thm3 2}
E_t = (G_t - \rho e_1) \cup (G_t + \rho e_1) 
\eeq
as long as the components $(G_t - \rho e_1)$ and $(G_t + \rho e_1)$ do not intersect each other. By the convexity of $G_t$,  the components $G_t - \rho e_1$ and  $G_t + \rho e_1$ do not intersect each other if the first one stays in the half-space $\{ x_1 < 0\}$ and the latter in $\{ x_1 >0 \}$. This is the same as  to say that the flow $G_t$ does not exit the strip $\{ -\rho < x_1 < \rho\}$. Let us show this.

Assume that for $h>0$ the family of sets $(G_t^h)_t$ is an approximate  flow obtained via \eqref{min mov} with the forcing term $f$ and starting from $G$.   We now show that  each $G^h_t$ is symmetric with respect to the coordinate axes and convex. Recall that the set $G^{h,1}$ is chosen as  a minimizer of the functional
$$
\mathcal{F}(E; G) =   P(E) + \frac{1}{h} \int_E \bar{d}_{G} \, dx  - \bar f(h) |E|.
$$
It is well known that in any dimension the function in \eqref{thm3 2} admits a minimal and a maximal minimizer which are convex and, by uniqueness, symmetric with respect to both coordinate axes, see \cite[Theorem 2]{BCCN}. However, in our two dimensional setting we can provide a simple self contained proof of this fact. 

Given $E$, we set  $E_+ = \{x \in E : x_1>0\}$ and $E_- = \{x \in E: x_1 <0\}$. By reflecting $E_+$ and $E_-$ with respect to the $x_2$-axis we obtain  sets $E_1$ and $E_2$, which are symmetric with respect to the $x_2$-axis and satisfy 
\[
\mathcal{F}(E_1; G)+ \mathcal{F}(E_2; G) \leq 2 \mathcal{F}(E; G).
\]
Then there exists $i=1,2$ such that $\mathcal{F}(E_i; G)\leq\mathcal{F}(E; G)$.
 By repeating the same argument with respect to the  $x_1$-direction we  conclude that we may choose $G^{h,1}$ symmetric with respect to both axes.

Let us show that $G^{h,1}$ is convex. By the Euler-Lagrange equation \eqref{euler} it holds
\[
\frac{\bar{d}_{G}}{h} = -k_{G^{h,1}} + \bar  f(h) \qquad \text{on } \, \pa G^{h,1}
\]
We claim that $\frac{\bar{d}_{G}}{h} (x) \leq\bar f(h)$ for all $x \in \pa G^{h,1}$. Indeed, suppose $x_0 \in \pa  G^{h,1}$ is the maximum of $\bar{d}_{G}$ on $\pa  G^{h,1}$.  If $\bar{d}_{G}(x_0) \leq 0$, then trivially $\frac{\bar{d}_{G}}{h} (x_0) \leq \bar f(h)$ as $f \geq 0$. If $\bar{d}_{G}(x_0) > 0$ then $x_0\not\in G$ and since it is  the furthest point from $G$ and $G$ is convex, it is easy to check that $k_{G^{h,1}}(x_0) \geq 0$. Then by the Euler-Lagrange equation 
\[
\frac{\bar{d}_{G}(x_0)}{h}  = -k_{G^{h,1}}(x_0) + \bar  f(h) \leq \bar  f(h) .
\] 
Therefore  $\bar{d}_{G}/h \leq\bar f(h)$ on  $ \pa G^{h,1}$ and by the  Euler-Lagrange equation
\[
k_{G^{h,1}} = - \frac{\bar{d}_{G}}{h} + \bar  f(h) \geq 0  \qquad \text{on } \, \pa G^{h,1}.
\]
Hence, $G^{h,1}$ is convex.

We now  apply to   $G^{h,1}$ the   Bonnesen circular symmetrization  with respect to the $x_2$-axis which, we recall,  decreases the perimeter, preserves  the area and decreases the dissipation term $\int_{G^{h,1}} \bar{d}_{G} \, dx $, by Lemma~\ref{bonnesen}.  Therefore we may assume that $G^{h,1}$ is   invariant under the  Bonnesen annular symmetrization with respect to the $x_2$-axis.  By iterating the  argument  we deduce that the same holds for $G_t^h$ for all $t>0$. Letting $h \to 0$ we deduce that the same holds for the flat flow, and  by the uniqueness  for the classical solution $(G_t)_t$. Therefore for every $t>0$ and $r>0$ the intersection  $G_t \cap \pa D_r$ is a union of two circular arcs with equal length which are both symmetric with respect to the $x_2$-axis. 

Now  if  $G_t$ exits the strip $\{ -\rho < x_1 < \rho\}$, say at time $t_0$, then the intersection $G_{t_0} \cap \pa D_\rho$ contains the points $(-\rho,0)$ and $(\rho,0)$. Since  $G_t \cap \pa D_\rho$ is a union of two circular arcs, which both are symmetric with respect to the $x_2$-axis,  we have  
\[
G_{t_0} \cap \pa D_\rho = \pa D_\rho.
\]
By the convexity of $G_{t_0}$ this implies  $D_\rho \subset G_{t_0}$. But since the flow \eqref{thm3 1} preserves the area  we have $|G_{t_0} |  = |D_\rho|$. Then  it holds   $G_{t_0} = D_\rho$, which is impossible. Therefore the flow  $G_t$ does not exit the strip $\{ -\rho < x_1 < \rho\}$,  \eqref{thm3 2} holds for all times and  the conclusion of the theorem follows. 

\end{proof}

\section*{Acknlowdgments}
The research of N.F. has been funded by PRIN Project 2015PA5MP7. The research of V.J. was supported by the Academy of Finland grant 314227. N.F. and M.M. are members of Gruppo Nazionale per l'Analisi Matematica, la Probabilit\`a e le loro Applicazioni (GNAMPA) of INdAM.

\end{document}